\documentclass[11pt, oneside]{amsart}   	
\usepackage{geometry}                		
\usepackage{graphicx}				
\usepackage{enumitem}

\usepackage{amssymb,amsthm}

\usepackage[utf8]{inputenc}

\usepackage{dsfont}  

\newtheorem{theorem}{Theorem}[section]
\newtheorem{lemma}[theorem]{Lemma}
\newtheorem{corollary}[theorem]{Corollary}
\newtheorem{prop}[theorem]{Proposition}
\theoremstyle{definition}

\theoremstyle{remark}
\newtheorem{rmk}{Remark}

\newcommand{\R}{{\mathds R}}

\newcommand{\B}{{\mathbb B}}

\newcommand{\grad}{\text{$\nabla$}}
\newcommand{\bs}[1]{{\boldsymbol  #1}}
\newcommand{\bb}[1]{{\mathbf  #1}}

\newcommand{\N}{N}
\newcommand{\HH}{{\mathcal H}}

\newcommand{\F}{{F}}

\DeclareMathOperator{\diam}{diam}

\usepackage{color}
\usepackage{hyperref}
\hypersetup{
	colorlinks=true,
	linkcolor=blue,
	citecolor=cyan,
	filecolor=magenta,      
	urlcolor=cyan,
	pdftitle={Overleaf Example},
	pdfpagemode=FullScreen,
}

\DeclareMathOperator{\rnge}{ran}

\DeclareMathOperator{\supp}{supp}

\title{Notes on  the discretization  of TV-norm regularized inverse potential problems.}

\author{L. Baratchart}
\address{Projet FACTAS, INRIA, 2004 route des Lucioles, BP 93,
	Sophia-Antipolis, 06902 Cedex, FRANCE}
\email{Laurent.Baratchart@inria.fr}

\author{D.P.~Hardin}
\address{Department of Mathematics,
	Vanderbilt University,
	Nashville, TN 37240, USA}
\email{doug.hardin@vanderbilt.edu}

\author{C. Villalobos Guill\'en}
\address{Projet FACTAS, INRIA, 2004 route des Lucioles, BP 93,
	Sophia-Antipolis, 06902 Cedex, FRANCE}
\email{cristobal.villalobos.guillen@protonmail.com}
 
\begin{document}
	\maketitle
















\noindent{\small{{\bf Abstract.} We describe a method to discretize optimization problems arising in the regularization of linear inverse problem having compact forward operator defined on  3-D valed measures, compactly supported on a fixed set. The criterion is a quadratic residual attached to the data, with an additive penalization of the total variation of the measure. }} 

\section{Notation and preliminaries}
Let $S$ be a compact set in $\R^\N$, $C(S)$  the real-valued continuous functions on $S$,    and $\mathcal{M}(S)$ the finite signed Borel measures on $\R^\N$ supported on $S$. For simplicity, we assume throughout that $S$ is uncountable.
Equipped  with the {\em total variation norm} $\|\cdot\|_{TV}$,  with
$\|\cdot\|$ to mean Euclidean norm, the set $\mathcal{M}(S)^\N$
of vector valued measures with $N$ components
is  a Banach space, dual to   $C(S)^\N$ under the pairing
\[
  \langle \bs{\varphi},\bs{\mu}\rangle=\int \bs{\varphi}\cdot d\bs{\mu},
  \qquad \bs{\varphi}\in C(S)^\N,\quad \bs{\mu}\in \mathcal{M}(S)^\N,
\]
where a dot indicates the Euclidean scalar product.
We  also consider $\mathcal{M}(S)^\N$ as a locally convex space  endowed with the weak-star topology, the dual of which is  $C(S)^\N$ \cite[Theorem 3.10]{Rudin}. Because  $C(S)^\N$ is separable, the weak-star topology restricted to
weak-star compact sets is metrizable \cite[Theorem 3.16]{Rudin}.


Let  $\HH$ be a real Hilbert space and $A:\mathcal{M}(S)^\N\to \HH$ a linear operator   sequentially continuous in the weak-star topology on $\mathcal{M}(S)^\N$.
	By the Banach-Alaoglu Theorem, any TV-norm bounded sequence contains a weak-star convergent subsequence and, hence, the sequential weak-star continuity of $A$  implies that $A$ is compact; i.e., $A$ maps TV-norm bounded sets into precompact sets in $\HH$.   

The focus of this paper is on the inverse source problem associated with $A$; that is, to estimate $\bs{\mu}$ given information about $A\bs{\mu}$.   It is well-known that the compactness of $A$ makes this an ill-posed problem and moreover, for our main motivation of magnetic source recovery, $A$ is not even injective.  It is classical to use regularization methods to produce `good' approximate solutions to $A\bs{\mu}=f$ by adjoining a penalty functional.
In \cite{BVHNS} and \cite{BVH},   we considered regularization using the TV-norm (total variation norm) for the case $A\bs{\mu}=\nabla\Phi\cdot \bs{v}$, where
$\bs{v}\in\R^N$ is a unit vector and $\Phi$ the solution to the Poisson equation $\Delta\Phi=\mathrm{div}\bs{\mu}$ vanishing at infinity; we showed this method is consistent for recovering a certain class of measures
including those carried on a countable set. Namely, it is consistent to recover measures supported on a purely 1-unrectifiable set, even though $A$ is not injective; compare \cite{BrePikk}.

Specifically, we consider for $f\in  \HH$ and $\lambda>0$,   the functional
\begin{equation}
\label{eq|defcrit0}
\mathcal{F}_{\lambda,f}(\bs{\mu}):=
\|f-A\bs{\mu}\|_{\HH}^2+\lambda \|\bs{\mu}\|_{TV}.
\end{equation}
If  $V$ is a weak-star closed subspace of $\mathcal{M}(S)^\N$, then there exists a minimizer $\bs{\mu}_{\lambda,f}^V$ (not necessarily unique) of 
  $\mathcal{F}_{\lambda,f}$ over $V$; i.e., 
  \begin{equation} \label{eq|mulamfVreg}
 \mathcal{F}_{\lambda,f}(\bs{\mu}_{\lambda,f}^V)= \inf_{\bs{\mu}\in V} \mathcal{F}_{\lambda,f}(\bs{\mu}).
 \end{equation}
	We write  $\bs{\mu}_{\lambda,f}$ for $\bs{\mu}_{\lambda,f}^V$ when $V=\mathcal{M}(S)^\N$.  We also find it convenient to define
\begin{equation}
\label{eq|defcrit1}
\F_{\lambda,f}(\bs{\mu}):=\mathcal{F}_{\lambda,f}(\bs{\mu})-\|f\|_{\HH}^2
=-2\langle f,A\bs{\mu}\rangle +\|A\bs{\mu}\|_{\HH}^2+\lambda \|\bs{\mu}\|_{TV},
\end{equation}
and note that  $\bs{\mu}$ is a minimizer in $V$ of $\F_{\lambda,f}$ if and only if it is a minimizer in $V$ of $\mathcal{F}_{\lambda,f}$.

\bigskip

The box below contains the definition and properties of $A$ for the particular case of inverse magnetization problems on a slender set; a slender set is one of zero measure whose complement has all its components of infinite measure.
	
\noindent\fbox{\parbox{\textwidth}{
	
Let $S\subset \R^3$ be   slender and compact.
Let $\rho$ be a positive measure  on $Q$ and let $A:\mathcal{M}(S)^3\to L^2(Q,\rho)$ be the \emph{forward operator} 
 mapping $\bs{\mu}$ to the restriction of $\bb{b}(\bs{\mu})\cdot v$ on $Q$. We adopt the setting of \cite{BVHNS} and remark from Proposition 3.1 in that paper, that $A$ is a compact operator. 
 
 For $\bs{\mu}\in \mathcal{M}(S)^3$ and
$v$ a unit vector in $\R^3$,  the component  of the magnetic field
$\bb{b}(\bs{\mu})$ in the direction $v$ at $x\not \in S$ 
is given by
\begin{equation*}\label{eq|b3K}
b_v(\bs{\mu})(x):=v\cdot\bb{b}(\bs{\mu})(x)  =-\frac{\mu_0}{4\pi}\int \bb K_v(x-y)\cdot \, d\bs{\mu}(y),
\end{equation*}
where 
\begin{equation*}
\label{eq|noyauK}
\bb K_v(x)=\frac{v}{|x|^3}  -3  x\frac{v\cdot x}{|x|^5}=  \grad \left( \frac{v\cdot x}{|x|^3}\right).
\end{equation*}

Consider  a finite, positive Borel measure  $\rho$
with support contained in $Q$  and let $A:
\mathcal{M}(S)^3\to L^2(Q,\rho)$ be the operator
defined by
\begin{equation*}\label{eq|Adef}
A(\bs{\mu})(x):= b_v(\bs{\mu})(x), \qquad x\in Q.
\end{equation*}  
Since $\bb{K}_v$ is continuous on $\R^3\setminus \{0\}$ and  $Q$ and $S$ are positively separated, it follows that $b_v$  
is continuous on $Q$ and consequently
$A$ does indeed map  $\mathcal{M}(S)^3$ into $L^2(Q,\rho)$. 
 
 Moreover, $A^*$ is mapping $L^2(Q,\rho)$ into $C(S)^3$ (endowed with the $L^\infty$-norm of the Euclidean norm), viewed as a subspace of the dual space $(\mathcal{M}(S)^3)^*$. In fact, $\mathrm{Im}\,A^*$ consists of restrictions to $S$ of real analytic functions on $\R^3\setminus Q$. 
 }}

\section{\textit{R}-Topology and GSM spaces}

Recall  the weak-star topology
on $\mathcal{M}(S)^\N$ and let $(\bs{\Phi}_n)_{n\in\N}$ be a dense sequence of functions in $C(S)^N$. The countable family of semi-norms $\bs{\mu}\mapsto \int\bs{\Phi}_n\cdot d\bs{\mu}$ endows $\mathcal{M}(S)^\N$ with a metrizable topology
$\tau_\Phi$, which is coarser that the weak-star topology but coincides with it on  weak-star compact sets; see proof of \cite[Theorem 3.16]{Rudin}.
We let $\mathfrak{d}$ be any 
metric compatible with this topology, see \cite[Section 3.8]{Rudin} for an example. 
We define on $\mathcal{M}(S)^\N$ the {\em $R$-distance} as
\[
\mathfrak{d}_R(\bs{\nu},\bs{\mu}):=\mathfrak{d}(\bs{\nu},\bs{\mu})+
\left|\,\|\bs{\nu}\|_{TV}-\|\bs{\mu}\|_{TV}  \,\right|,
\qquad \bs{\nu},\bs{\mu}\in\mathcal{M}(S)^\N, 
\]
and call the corresponding topology the {\em $R$-topology} (we write $\stackrel{R}{\longrightarrow}$ to denote convergence in the $R$-topology).
A set $W\subset\mathcal{M}(S)^\N$ is {\em $R$-dense in $\mathcal{M}(S)^\N$} if, for any $\bs{\mu}\in\mathcal{M}(S)^\N$, there is a sequence $\bs{\mu}_n$ such that $\bs{\mu}_n$ converges weak-star to $\bs{\mu}$ and  $\|\bs{\mu}_n\|_{TV}$  converges to $\|\bs{\mu}\|_{TV}$ as $n\to\infty$.
This follows from the following lemma,
showing that the \textit{R}-topology is the metrizable topology induced by sequential weak-star convergence of measures together with weak-star convergence of their  total variation measures.  

\begin{lemma}\label{lemma|equivTopo}
The sequence $\bs{\mu}_n\in\mathcal{M}(S)^\N$ converges in the \textit{R}-topology to $\bs{\mu}\in\mathcal{M}(S)^\N$ if and only if $\bs{\mu}_n$ converges weak-star to $\bs{\mu}$ and $|\bs{\mu}_n|$ converges weak-star to $|\bs{\mu}|$.
\end{lemma}
\begin{proof}
  Taking a smooth compactly supported function equal to 1 on $S$, it is clear that weak-star convergence of $|\bs{\mu}_n|$ to $|\bs{\mu}|$ implies convergence of $\|\bs{\mu}_n\|_{TV}$ to $\|\bs{\mu}\|_{TV}$. Then 
  $\|\bs{\nu}_n\|_{TV}$ is  {\it a fortiori} bounded,  so
  $\bs{\mu}_n$ and $\bs{\mu}$ are
  all contained in a closed ball of $\mathcal{M}(S)^\N$ on which
  weak-star convergence of  $\bs{\mu}_n$ to $\bs{\mu}$ means
  that $\lim_n\mathfrak{d}(\bs{\mu}_n-\bs{\mu})=0$, by definition of $\mathfrak{d}$. Hence the ``if'' part of the statement of the lemma follows.
  Conversely, assume that $\bs{\mu}_n$ converges in the \textit{R}-topology to $\bs{\mu}$, whence in particular  $|\bs{\mu}_n|\to |\bs{\mu}|$.
  Again $\bs{\mu}_n$ and $\bs{\mu}$ lie
  in a closed ball of $\mathcal{M}(S)^\N$, and 
by definition of $\mathfrak{d}$ we get weak-star convergence of $\bs{\mu}_n$ to $\bs{\mu}$. To achieve the proof, it remains to establish that $|\bs{\mu}_n|$ converges weak-star to $|\bs{\mu}|$.
	Since the $|\bs{\mu}_n|$ are bounded, by the Banach-Alaoglu theorem there exists a subsequence $|\bs{\mu}_{n_m}|$ weak-star convergent to some positive measure $\mu^*$.
	
	As we next show,  it follows from the weak-star convergence of $\bs{\mu}_n$ and $|\bs{\mu}_{n_m}|$ that, for every measurable set $W\subset S$, $|\bs{\mu}|(W)\leq\mu^*(W)$.
		Indeed, by regularity of Radon measures, we may  assume that    $W$ is compact.
	Let $\bb{u}_{\bs{\mu}}$ denote the Radon-Nikodym derivative of $\bs{\mu}$ with respect to $|\bs{\mu}|$.
	Take an $\epsilon>0$ and use Lusin's theorem to find a compact set $E_\epsilon$ and a function $\bb{u}_\epsilon\in C_c(S)^\N$ such that $|\bs{\mu}|(S\setminus E_\epsilon)<\epsilon$, $\bb{u}_\epsilon=\bb{u}_{\bs{\mu}}$ on $E_\epsilon$ and $|\bb{u}_\epsilon|\leq 1$ on $S$.
	By Urysohn's Lemma, there exists a smooth function $\phi$ valued in $[0,1]$ such that 
$\phi$ equals 1 on $W$  and both $\mu^*($supp$(\phi)\setminus W)$ and $|\bs{\mu}|($supp$(\phi)\setminus W)$ are less than $\epsilon$.
	Now,
\begin{align}\nonumber
	|\bs{\mu}|(W)-2\epsilon &< |\bs{\mu}|(E_\epsilon\cap W)-\epsilon
	<\langle\bs{\mu},\phi\bb{u}_\epsilon\rangle
	=\lim_{n\to\infty}\langle \bs{\mu}_n,\phi\bb{u}_\epsilon\rangle
	\\
	&\leq\lim_{m\to\infty}\langle|\bs{\mu}_{n_m}|,\phi\rangle
	=\langle \mu^*,\phi\rangle< \mu^*(W)+\epsilon.\label{inegReps}
\end{align}
	
	Now, taking again a smooth compactly supported function equal to 1 on $S$, the convergence of $\|\bs{\mu}_n\|_{TV}$ to $\|\bs{\mu}\|_{TV}$ implies that $\|\mu^*\|_{TV}=\|\bs{\mu}\|_{TV}$.
	Thus, since \eqref{inegReps} implies that for every measurable set $W\subset S$ holds  $|\bs{\mu}|(W)\leq\mu^*(W)$, we get that $\mu^*=|\bs{\mu}|$ and therefore $|\bs{\mu}_n|$ converges weak-star to $|\bs{\mu}|$.
\end{proof}

We next introduce a family of linear subspaces of $\mathcal{M}(S)^\N$ that will be used for approximating minimizers of \eqref{eq|defcrit1}.
Let $\mathcal{N}=\{\nu_k\}_{k\in I}$  be a denumerable family  of mutually singular Borel probability measures on $S$   with index set $I\subset {\mathbb N}$ and let $V=V(\mathcal{N})$ denote the space of measures $\bs{\mu}\in \mathcal{M}(S)^\N$ of the form
\begin{equation}
  \label{eq|defGSM}
  \bs{\mu}= \sum_{k\in I} {\bf m}_k \nu_k,
  \end{equation}
for ${\bf m}_k\in \R^\N$, $k\in I$. 
Note that $\|\bs{\mu}\|_{TV}= \sum_{k\in I} |{\bf m}_k | <\infty$.
We refer to $V$ as a {\em generated by  singular measures (GSM)} space and we will refer to the $\nu_k$'s as the {\it generating measures} of $V$.
Let ${\mathcal E}=\{E_k\}_{k\in I}$  be a Borel partition of $S$ such that $\nu_k(E_j)=\delta_{ij}$ for $i,j\in I$ (such partitions exist due to the pairwise mutual singularity of the $\nu_k$'s).
 We call such a partition \emph{compatible with $V$.}
For $\bs{\mu}\in {\mathcal M}(S)^\N$ we define the {\em projection of $\bs{\mu}$ onto $V$ with respect to ${\mathcal E}$} as
$$P^{\mathcal{E}}_V(\bs{\mu}):=\sum_{k\in I}\bs{\mu}(E_k)\nu_k.$$

We remark that if the positive cone of $\mathcal{N}=\{\nu_k\}_{k\in I}$  is   weak-star dense   in $\mathcal{M}(S)^+$, then $V(\mathcal{N})$ is $R$-dense 
in  $\mathcal{M}(S)^\N$. This is easily seen using the Hahn decomposition of a signed mesure in mutually sigular positive and negative parts.


\begin{rmk}\label{rmk|realization}
	V as above is isometrically isomorphic to the Banach space 
	$$\left\{({\bf m}_k)_{k\in I}\in(\R^\N)^I  :  \|({\bf m}_k)_{k\in I}\|_{1} <\infty\right\},$$
	where $\|({\bf m}_k)_{k\in I}\|_{1}:=\sum_{k\in I} |{\bf m}_k |$, and so its dual $V^*$ is isometrically isomorphic to the Banach space
	$$\left\{({\bf c}_k)_{k\in I}\in(\R^\N)^I : \|({\bf c}_k)_{k\in I}\|_{\infty}<\infty\right\}$$
	where $\|({\bf c}_k)_{k\in I}\|_{\infty}:=\sup_{k\in I} |{\bf c}_k |$.
\end{rmk}

\begin{lemma}\label{lemma|prevexample}
  For $n\in \mathbb{N}$, let $V_n$ be a   GSM space  with generating measures $\mathcal{N}_n=\{\nu_k^n\}_{k\in I_n}$ and let ${\mathcal E^n}=\{E^n_k\}_{k\in I_n}$ be a  Borel partition of $S$ compatible with $V_n$. Let us define
    $a_n:= \sup_{k\in I_n}\diam  E^n_k$.
	If $a_n\to 0$ as $n\to\infty$, then for any $\bs{\mu}\in\mathcal{M}(S)^\N$ the sequence $\bs{\mu}_n:=P_{V_n}^{\mathcal{E}^n}(\bs{\mu})$ converges to $\bs{\mu}$ in the $R$-topology and, hence,  $\bigcup_n V_n$ is $R$-dense in $\mathcal{M}(S)^\N$.
\end{lemma}

\begin{rmk}
The existence of a sequence of compatible partitions $\{E^n_k\}_{k\in I_n}$ as above such that $\lim_{n\to\infty} a_n= 0$ can be expressed in terms of the supports $F^n_k=\text{supp} (\nu_k^n)$, $k\in I_n$, $n\in \mathbb{N}$; namely, the existence of such a sequence of partitions is equivalent to  the conditions
\begin{itemize}
\item $\lim_{n\to\infty}\sup_{k\in I_n} \diam F^n_k=0$ and
\item  $\lim_{n\to \infty} \sup_{x\in S}d(x,\bigcup_{k\in I_n}F^n_k)=0$.  
\end{itemize} We remark that  $\sup_{x\in S}d(x,\bigcup_{k\in I_n}F^n_k)$ is the {\em covering radius} of $\bigcup_{k\in I_n}F^n_k$ relative to $S$.

\end{rmk}
\begin{proof}
	Take any $\bs{\mu}\in\mathcal{M}(S)^\N$, and let $\bs{\mu}_n:=P_{V_n}^{\mathcal{E}^n}(\bs{\mu})$.
	By using countable additivity twice we have that
	$$
	\langle \bs{\mu}_n,\bs{\phi}\rangle 
	=\sum_{k\in I_n}\left\langle \langle\bs{\mu},\chi_{E_k^n}\rangle\nu_k^n,\bs{\phi}\right\rangle
	=\left\langle\bs{\mu},\sum_{k\in I_n}\langle \nu_k^n,\bs{\phi}\rangle\chi_{E_k^n}\right\rangle,
	$$
	that is, $\langle \bs{\mu}_n
,\bs{\phi}\rangle$ is equal to the integral against $\bs{\mu}$ of the simple function $f_n:=\sum_{k\in I_n}\langle \nu_k^n,\bs{\phi}\rangle\chi_{E_k^n}$.
	Now, for every $x\in S$ and $n\in\mathbb{N},$ let $k_n(x)\in I_n$ be such that $E_{k_n(x)}^n$ contains $x$.
	Then, $$|f_n(x)-\bs{\phi}(x)|\leq \max\{|\bs{\phi}(y)-\bs{\phi}(x)| : y\in E_{k_n(x)}^n\}.$$
	Hence, the continuity of $\bs{\phi}$ implies that $f_n\to\bs{\phi}$ pointwise.
	Since $|f_n|$ is dominated by the $|\bs{\mu}|$-integrable constant function $\max|\bs{\phi}|$, then $\bs{\mu}_n\to\bs{\mu}$ weak-star, by dominated convergence.
	
	From the Banach-Alaoglu theorem, it follows that $\liminf_{n\to\infty}\|\bs{\mu}_n\|_{TV}\geq\|\bs{\mu}\|_{TV}$.
Since by definition 
$$\|\bs{\mu}_n\|_{TV}=\sum_{k\in I_n}|\bs{\mu}(E^n_k)|\leq \sum_{k\in I_n}|\bs{\mu}|(E^n_k)=\|\bs{\mu}\|_{TV},$$ 
then $\|\bs{\mu}_n\|_{TV}\to\|\bs{\mu}\|_{TV}$.
By Lemma \ref{lemma|equivTopo} this implies that $\bs{\mu}_n\to\bs{\mu}$ in the $R$-topology, and therefore $\bigcup_nV_n$ is $R$-dense in $\mathcal{M}(S)^\N$.
\end{proof}

\section{Critical point equations of $\mathcal{F}_{\lambda, f}$}

Recall that $A:\mathcal{M}^\N(S)\to \HH$ is sequentially weak-star continuous.
Regardless of the context, $A^*$ will refer to the dual of $A:\mathcal{M}(S)^N\to\HH$ and thus maps
$\HH$ into $C(S)^\N$.

\begin{theorem}\label{thm|CPV}
	Let $V\subset\mathcal{M}(S)^\N$ be a linear space closed under the $\|\cdot\|_{TV}$ norm, $\lambda>0$ and $f\in  \HH$.
	A measure $\bs{\mu}\in V$ is a minimizer of $\mathcal{F}_{\lambda, f}$ over $V$  (i.e., $\bs{\mu}=\bs{\mu}_{\lambda,f}^V$ as in \eqref{eq|mulamfVreg} if such a minimizer exists)
	if and only if
	\begin{equation}
		\label{eq|generalCPV}
		\begin{array}{ccll}
			\langle \bs{\mu} , A^*(f-A\bs{\mu})\rangle &=&\frac{\lambda}{2} \|\bs{\mu}\|_{TV}, &  
			\text{ and}\\
			\left |\langle \bs{\nu} , A^*(f-A\bs{\mu})\rangle\right| &\leq& \frac{\lambda}{2}\|\bs{\nu}\|_{TV},&\forall \bs{\nu}\in V.
		\end{array}
	\end{equation}
	Moreover, $\bs{\mu}'\in V$ is another minimizer of $\mathcal{F}_{\lambda, f}$ over $V$ if and only if 
\begin{enumerate}[label=(\alph*)]
\item $A(\bs{\mu}'-\bs{\mu})=0$,
\item and $\langle \bs{\mu}' , A^*(f-A\bs{\mu})\rangle =\frac{\lambda}{2} \|\bs{\mu}'\|_{TV}$.
\end{enumerate}
\end{theorem}
\begin{proof}
	For a convex functional $\mathcal{G}$ on $V$ and a $\bs{\mu}\in V$, let $\partial\mathcal{G}(\bs{\mu})$ denote the subdifferential of $\mathcal{G}$ at $\bs{\mu}\in V$, that is, $\partial\mathcal{G}(\bs{\mu}):=\{\mathcal{N}\in V^*\ :\ \mathcal{G}(\bs{\nu})-\mathcal{G}(\bs{\mu})\geq \langle\mathcal{N},\bs{\nu}-\bs{\mu}\rangle,$ for any $\bs{\nu}\in V\}$.
	Let $\mathcal{R}$ be the convex functional on $V$ that sends 
$\bs{\mu}$ to $ \|f-A\bs{\mu}\|_{\HH}^2$, and note that $\mathcal{R}$ has a Fr\'echet derivative, $\grad\mathcal{R}$, such that for any $\bs{\nu}\in V$, $\langle \grad\mathcal{R},\bs{\nu}\rangle= \langle\bs{\nu},2A^*(A\bs{\mu}-f)\rangle$.
	Let $\|\cdot\|_{V}$ indicate the restriction of the total variation norm to $V$ and observe that
	$$\partial(\|\cdot\|_{V})(\bs{\mu})=
	\left\{\begin{array}{lr}
		\left\{\mathcal{N}\in V^*\ :\ \|\mathcal{N}\|_{V^*}=1 \text{ and } \langle\mathcal{N},\bs{\mu}\rangle=\|\bs{\mu}\|_{TV}\right\} & \text{if } \bs{\mu}\neq0 \\
		\left\{\mathcal{N}\in V^*\ :\ \|\mathcal{N}\|_{V^*}\leq1\right\} & \text{if } \bs{\mu}=0
	\end{array}\right.,$$ 
	where $\|\mathcal{N}\|_{V^*}$ represents the dual norm of $\mathcal{N}$ with respect to $\|\cdot\|_{V}$.
	Then, by properties of subdifferentials \cite[Chapter III, section 2]{ekeland_turnbull_1983}, we obtain:
	$$
	\partial\mathcal{F}_{\lambda,f}(\bs{\mu}) = \{\grad\mathcal{R}\} + \lambda\partial(\|\cdot\|_{V})(\bs{\mu}).
	$$
	Thus, since $\bs{\mu}$ being a minimizer of $\mathcal{F}_{\lambda,f}$ over $V$ is equivalent to $0\in \partial\mathcal{F}_{\lambda,f}(\bs{\mu})$,
	we have that $\bs{\mu}$ is a minimizer of $\mathcal{F}_{\lambda, f}$ over $V$ if and only if ${\frac{1}{\lambda}\grad\mathcal{R}}\in{\partial(\|\cdot\|_{V})(\bs{\mu})}$.
	Therefore, $\bs{\mu}$ is a minimizer of $\mathcal{F}_{\lambda, f}$ over $V$ if and only \eqref{eq|generalCPV} is satisfied.
	
	Now take $\bs{\mu}\in V$ to be a minimizer of $\mathcal{F}_{\lambda, f}$ over $V$ and let $\bs{\mu}'\in V$.
	Assume that $\bs{\mu}$ and $\bs{\mu}'$ satisfy {\it(a)} and {\it(b)}.
	Then, since $\bs{\mu}$ satisfies the second equation of \eqref{eq|generalCPV} and {\it(a)} holds, then it follows that the second equation of \eqref{eq|generalCPV} holds for $\bs{\mu}'$ as well.
	Also, combining {\it(a)} and {\it(b)}, we get the first equation of \eqref{eq|generalCPV} for $\bs{\mu}'$.
	
	Finally, assume that  $\bs{\mu}'$ is a minimizer of $\mathcal{F}_{\lambda, f}$ over $V$.
	By convexity of norms 
	and of $\mathcal{F}_{\lambda, f}$, and the minimality of $\bs{\mu}$ and $\bs{\mu}$', we get that $\|f-A\bs{\mu}\|_2=\|f-A\bs{\mu}'\|_2$.
	Then, using the	strict convexity of $\|\cdot\|_2$, it follows that $f-A\bs{\mu}=f-A\bs{\mu}'$, hence {\it(a)}.
	To finish, note that {\it(b)} follows from {\it(a)} and the first equation of \eqref{eq|generalCPV}.
\end{proof}

Theorem \ref{thm|CPV} says that $\bs{\mu}\in V$  is a minimizer of \eqref{eq|mulamfVreg} if and only if 
integration against
$2A^*(f-A\bs{\mu})/\lambda$ is a norming functional for $\bs{\mu}$ in $V$.
Recalling that GSM spaces are Banach spaces isometric to a subspaces of $l^1(\R^\N)$
(see remark \ref{rmk|realization}), we obtain the following corollary.

\begin{corollary}\label{coro|CPV}
  Let $V\subset\mathcal{M}(S)^\N$ be a GSM space, $\lambda>0$ and $f\in  \HH$. 
  A measure $\bs{\mu}=\sum_{k\in I} {\bf m}_k \nu_k\in V$ is a minimizer of $\mathcal{F}_{\lambda, f}$ over $V$  (i.e., $\bs{\mu}=\bs{\mu}_{\lambda,f}^V$ as in \eqref{eq|mulamfVreg})
if and only if
\begin{equation}
\label{eq|CPV}
\begin{array}{ccll}
\langle \nu_k , A^*(f-A\bs{\mu})\rangle &=&\frac{\lambda}{2} \frac{\mathbf{m}_k}{|\mathbf{m}_k|}, & \forall k\in I \text{ s.t. } \mathbf{m}_k\neq 0, 
\text{ and}\\
\left |\langle \nu_k , A^*(f-A\bs{\mu})\rangle\right| &\leq& \frac{\lambda}{2},&\forall k\in I.
\end{array}
\end{equation}
Moreover, $\bs{\mu}'=\sum_{k\in I} {\bf m}'_k \nu_k\in V$ is another minimizer of $\mathcal{F}_{\lambda, f}$ over $V$ if and only if 
\begin{enumerate}[label=(\alph*)]
\item $A(\bs{\mu}'-\bs{\mu})=0$,
\item and there exist $p_k\geq0$ for every $k\in I$ such that $p_k=0$ if $\left |\langle \nu_k , A^*(f-A\bs{\mu})\rangle\right| < \frac{\lambda}{2}$ and
\[
\bb{m}'_k=\left\{\begin{array}{lr}
	p_k\frac{\bb{m}_k}{|\bb{m}_k|}  &  \text{if }\ \mathbf{m}_k\neq 0  \\
	p_k\frac{2}{\lambda}\langle \nu_k , A^*(f-A\bs{\mu}) \rangle  & \text{otherwise.}
\end{array}\right.
\]
\end{enumerate}
\end{corollary}
\begin{proof}
	Following our previous discussion what is left to prove is that, when taking the main assumptions of the corollary and {\it(a)}, then {\it(b)} of Corollary~\ref{coro|CPV} is equivalent to {\it(b)} of Theorem~\ref{thm|CPV}.
	This equivalence can be seen noticing that, in view of the second equation of $\eqref{eq|CPV}$, the inequality
	\begin{align*}
		\frac{\lambda}{2}\sum_{k\in I}\bb{m}'_k\cdot\frac{\bb{m}'_k}{|\bb{m}'_k|} &=\frac{\lambda}{2}\|\bs{\mu}'\|_{TV}
		\ \geq\  \langle \bs{\mu}' , A^*(f-A\bs{\mu})\rangle 
		= \sum_{k\in I} \langle \bb{m}'_k\nu_k , A^*(f-A\bs{\mu}) \rangle\\
		&=\sum_{{\tiny \begin{array}{c}	k\in I\\\mathbf{m}_k\neq 0\end{array}}} \bb{m}'_k\cdot\left(\frac{\lambda}{2}\frac{\bb{m}_k}{|\bb{m}_k|}\right) 
		+ \sum_{{\tiny \begin{array}{c}	k\in I\\\mathbf{m}_k=0\end{array}}} \bb{m}'_k\cdot\langle \nu_k , A^*(f-A\bs{\mu}) \rangle
	\end{align*}
	is an equality if and only if the following two conditions are satisfied:
	\begin{itemize}
		\item if $\mathbf{m}_k\neq0$, then $\bb{m}'_k$ is parallel to $\bb{m}_k$;
		\item if $\mathbf{m}_k=0$, then $\bb{m}'_k$ is parallel to $\langle \nu_k , A^*(f-A\bs{\mu}) \rangle$ and $|\langle \nu_k , A^*(f-A\bs{\mu}) \rangle|=\frac{\lambda}{2}$.
	\end{itemize}
\end{proof}

Note that Theorem \ref{thm|CPV} and Corollary~\ref{coro|CPV} do not assume that a minimizer  to
$\mathcal{F}_{\lambda,f}$ over $V$ exists.
When $V$ is weak-star closed, such a minimizer does exist and
the significance  of the critical point equation \eqref{eq|generalCPV}
can be deepened. Consider the subspace of $C(S)^\N$
defined by  $V^\perp:=\{\bb{p}\in C(S)^\N:\langle\bs{\nu},\bb{p}\rangle=0,\ \forall\bs{\nu}\in V\}$, and let 
$\|\bb{f}\|_{\infty}:=\max_{x\in S}|\bb{f}(x)|$ indicate
the norm of $\bb{f}\in C(S)^\N$. 
	Since $V$ is weak-star closed in $\mathcal{M}(S)^\N$,
the Hahn-Banach theorem (valid in any locally convex topological space)   implies that the space of those  
$\bs{\nu}\in\mathcal{M}(S)^\N$ such that $\langle\bs{\nu},\bf{f}\rangle=0$ for all
${\bf{f}}\in V^\perp$ is equal to $V$. Thus,
we get from the Hahn-Banach theorem again that $V$ is isometrically isomorphic to $(C(S)^\N/V^{\perp})^*$, where the quotient space $C(S)^\N/V^{\perp}$
is equipped with the norm  $\|[\bb{f}]\|_{\infty/\perp}:=\inf_{\bb{p}\in V^\perp}\|\bb{p}+\bb{f}\|_\infty$ for $[\bb{f}]\in C(S)^\N/V^{\perp}$, see \cite[Theorem 7.2]{duren_1970}.

	\begin{corollary}
\label{coro|densopt}
Let $V\subset\mathcal{M}(S)^3$ be a weak-star closed linear space, $\lambda>0$ and $f\in  \HH$.
Then, $\bs{\mu}$ is a minimizer in \eqref{eq|mulamfVreg}
if and only if there exists a sequence $\{\bb{p}_n\}\subset V^\perp$,
$n\in\mathbb{N}$, for which 
	\begin{equation}\label{eq|CPquotient}
			\begin{array}{ccll}
		\lim_{n\to\infty}	A^*(f-A\bs{\mu})+\bb{p}_n&=&\frac{\lambda}{2} \bb{u}_{\bs{\mu}}\quad \text{in}\quad L^1(|\bs{\mu}|)^\N,&  
			\text{ and}\\
		\lim_{n\to\infty}\left\|\ A^*(f-A\bs{\mu})+\bb{p}_n\ \right\|_{\infty}&=& \frac{\lambda}{2}.
		\end{array}
	\end{equation}
In the case that $V$ is a GSM space with generating measures $\{\nu_n\}_{k\in I}$, there is $\bb{p}\in V^\perp\subset L^\infty(\sum_k\nu_k)^\N$ such that
\begin{equation*}\label{eq|CPquotienti}
			\begin{array}{ccll}
		A^*(f-A\bs{\mu})+\bb{p}&=&\frac{\lambda}{2} \bb{u}_{\bs{\mu}}\quad |\bs{\mu}|-\text{a.e.},&  
			\text{ and}\\
		\left\|\ A^*(f-A\bs{\mu})+\bb{p}\ \right\|_{L^\infty(S)^3}&\leq& \frac{\lambda}{2}.
		\end{array}
	\end{equation*}
\end{corollary}
\begin{proof}
From Theorem \ref{thm|CPV}, we know that $\bs{\mu}$ is a minimizer of 
$\mathcal{F}_{\lambda,f}$ over $V$ if and only integration against
$2A^*(f-A\bs{\mu})/\lambda$ is a norming functional for $\bs{\mu}$ in $V$.  
By the isometric identification of $V$ with  $(C(S)^3/V^{\perp})^*$ and the Radon-Nykodim decomposition $d\bs{\mu}=\bf{u}_{\bs{\mu}}d|\bs{\mu}|$,  
it means that some sequence $A^*(f-A\bs{\mu})+\bb{p}_n$
from  the coset $[A^*(f-A\bs{\mu})]\in C(S)^3/V^{\perp}$ satisfies:
\begin{itemize}
\item[(i)] $\lim_n \|A^*(f-A\bs{\mu})+\bb{p}_n\|_\infty\longrightarrow\lambda/2$,
\item[(ii)] $\int \Bigl(\frac{\lambda}{2}-\bigl(A^*(f-A\bs{\mu})+\bb{p}_n\bigr)
\cdot\bb{u}_{\bs{\mu}}\Bigr)\,
d|\bs{\mu}|=0$ for all $n$.
\end{itemize}
From $(i)$ we get that the integrand in $(ii)$ has a $\liminf$ which is
nonnegative $|\bs{\mu|}$-a.e., hence it must go to zero 
$|\bs{\mu}|$-a.e., for if the set $S_\eta:=\{t\in S:\,\lambda/2-(A^*(f-A\bs{\mu})+\bb{p}_n)(t)\cdot\bb{u}_{\bs{\mu}}(t)>\eta\}$ had $|\bs{\mu}|$-measure greater than $\varepsilon>0$ for some $\eta>0$ and infinitely many $n$, then $(ii)$
would contradict  that $\liminf_n\int_{S\setminus S_\eta}(\frac{\lambda}{2}-(A^*(f-A\bs{\mu})+\bb{p}_n)\cdot\bb{u}_{\bs{\mu}})d|\bs{\mu}|\geq0$ by the Fatou lemma.
By dominated convergence, we now get that
$A^*(f-A\bs{\mu})+\bb{p}_n$
tends to $(\lambda/2) \bb{u}_{\bs{\mu}}$ in $L^1(|\bs{\mu}|)$.

Assume next that $V$ is a GSM space with generating measures $\{\nu_n\}_{k\in I}$.
By a diagonal argument,
we may assume that 
$\bb{p}_n\to \frac{\lambda}{2} \bb{u}_{\bs{\mu}}-A^*(f-A\bs{\mu})$
pointwise $|\bs{\mu}|$-a.e. and that $\bb{p}_n$ converges to
$\bb{p}\in L^\infty(\sum_k\nu_K)^3$,
weak-star as $n\to \infty$ ($L^\infty(\sum_k\nu_K)^3$ being viewed as the dual of $L^1(\sum_k\nu_K)^3$). Redefining $\bb{p}$ as $-A^*(f-A\bs{\mu})$
on the complement of a Borel carrier for $\sum_k\nu_k$ containing the
Lebesgue points of  $\bb{p}$ with respect to  $\sum_k\nu_k$,
we find that this $\bf{p}$ satisfies all our requirements.

\end{proof}

\begin{rmk}
If $V=\mathcal{M}(S)^\N$, then \eqref{eq|CPquotient} reduces to 
the familiar optimality condition for quadratic criteria on
 measures regularized by the total variation ({\it cf} for instance \cite{BrePikk}); 
\begin{equation}
\label{eq|CP}
\begin{array}{ll}
A^*(f-A\bs{\mu})&=\frac{\lambda}{2} \bb{u}_{\bs{\mu}} 
\qquad |\bs{\mu}|\text{\rm -a.e. and}\\
\left |A^*(f-A\bs{\mu})\right|&\leq \frac{\lambda}{2}\quad\text{\rm everywhere on } S.
\end{array}
\end{equation}
When $V$ consists of (finite or infinite)
sums of dipoles with summable moments (i.e. a $\R^\N$-vector times a Dirac measure) located at points $\{x_1,x_2,\cdots\}$ of $S$, that is:
\[
V=\{\sum_k{\bf m}_k\delta_{x_k}:\ \sum_k|{\bf m}_k|<\infty\},
\]
then the situation is clearly of the GSM type and \eqref{eq|CPV}
means that an optimal $\bs{\mu}$ makes
$A^*(f-A\bs{\mu})$ have norm at most $\lambda/2$ at each $x_k$ and interpolate
$(\lambda/2)\bs{m}_k/|{\bf m}_k|$ at $x_k$ whenever ${\bf m}_k\neq0$;
if, moreover, $V$ is weak-star closed (in particular if the $x_k$ are finite in number), then such a $\bs{\mu}$ is guaranteed to exist.
In another connection, if $V$ consists of absolutely continuous measures with respect to Lebesgue measure $m$,
whose density is constant on measurable
subsets  $\{E_1,E_2,\cdots\}$ of $S$ such that $m(E_{j_1}\cap E_{j_2})=0$
if $j_1\neq j_2$  and is zero elsewhere:
\[
V=\{\sum_k{\bf m}_k\chi_{E_k}dm:\ \sum_k|{\bf m}_k|m(E_k)<\infty\},
\]
then \eqref{eq|CPquotient}
tells us  that an optimal $\bs{\mu}$
makes the mean of $A^*(f-A\bs{\mu})$ on $E_k$ of norm at most $\lambda/2$
and equal to $(\lambda/2)\bs{m}_k/|{\bf m}_k|$ whenever ${\bf m}_k\neq0$.
\end{rmk}

%
%
%


\section{Some inequalities}

For $\bs{\mu}\in\mathcal{M}(S)^\N$, $V\subset \mathcal{M}(S)^\N$ (typically $V$ is a subspace)  and $g\in  \HH$, we define
\begin{equation*}
\label{eq|defkappamu}
\kappa (V,\bs{\mu}):=\inf_{\substack {\bs{\nu}\in V}}\max\left\{\|A(\bs{\mu}-\bs{\nu})\|_{\HH}\ ,\ \left|\,\|\bs{\mu}\|_{TV}-\|\bs{\nu}\|_{TV}  \,\right|  \right\}, 
\end{equation*}
and 
\begin{equation*}\label{eq|deltan}
\delta(g,V):=\sup_{\substack {\bs{\mu}\in V\\ \|\bs{\mu}\|_{TV}=1}}
 |\langle g,A\bs{\mu}\rangle | \le \|A^*g\|_\infty.  
 \end{equation*}

 Subsequently,  for $W\subset \mathcal{M}(S)^\N$ a subspace, 
we put
\begin{equation*}
\label{eq|defkappa}
\kappa (W,V):=\sup_{\substack {\bs{\mu}\in W\\ \|\bs{\mu}\|_{TV}=1}}\ \ \kappa(V,\bs{\mu}),\ \mathrm{ and\ simply}\ \kappa (V):=\kappa( \mathcal{M}(S)^\N,V).
\end{equation*}
We remark that $\kappa(V,\bs{\mu})\le \kappa(V)\|\bs{\mu}\|_{TV}$ for $\bs{\mu}\in \mathcal{M}(S)^\N$.

If $V_1\subset V_2\subset\cdots$ is a nested sequence of subsets of 
$\mathcal{M}(S)^\N$ whose union is $R$-dense, it is clear that
$\lim_n\kappa(V_n,\bs{\mu})=0$ for each $\bs{\mu}\in \mathcal{M}(S)^\N$.

 For $\bs{\mu}\in V\subset \mathcal{M}(S)^\N$ and
$\tilde{f},f\in \HH$, then
\begin{equation}\label{eq|ineq_delta}
\left| \F_{\tilde{f},\lambda}(\bs{\mu})-\F_{ f,\lambda}(\bs{\mu})\right|
=2|\langle f-\tilde{f},A\bs{\mu}\rangle|\le 2\delta(f-\tilde{f},V)\|\bs{\mu}\|_{TV}.
\end{equation}
Furthermore, for $\bs{\mu},\bs{\nu}\in\mathcal{M}(S)^\N$ and  $\lambda>0$ we get from \eqref{eq|defcrit1} that
\begin{equation}
\begin{split}
  \F_{ f,\lambda}(\bs{\mu})-\F_{ f,\lambda}(\bs{\nu}) 
&= 2\langle f,A(\bs{\nu}-\bs{\mu})\rangle  + \langle A(\bs{\mu}+\bs{\nu}), A(\bs{\mu}-\bs{\nu}) \rangle +\lambda \left(\|\bs{\mu}\|_{TV}-\|\bs{\nu}\|_{TV}\right). 
\end{split}
\end{equation}
In the following lemma we shall find the following definition useful:
  \begin{equation}
\label{eq|defd}
 d_\lambda(f,\tilde{f},V,\bs{\mu}):=\kappa(V,\bs{\mu})(2\|\tilde{f}\|_{\HH}+4\|f\|_{\HH}+\kappa(V,\bs{\mu})+\lambda).
\end{equation}
 
\begin{lemma} \label{lemma|ineq_F_2f_V} Let $V$ be a weak-star closed subspace of $ \mathcal{M}(S)^\N$, $\tilde{f},f\in \HH$, and $\lambda>0$.
Then
\begin{equation}
\label{eq|fncond3}
\F_{\tilde{f},\lambda}(\bs{\mu}_{\lambda,\tilde{f}}^V)
\le \F_{\tilde{f},\lambda}(\bs{\mu}_{\lambda,f})+d_{\lambda}(f,\tilde{f},V,\bs{\mu}_{\lambda,f}) 
\end{equation}
and
\begin{equation}\begin{split}\label{eq|fncond2}
	-2\delta(\tilde{f}-f,V)\|\bs{\mu}_{\lambda,\tilde{f}}^V\|_{TV}	&\le \F_{\tilde{f},\lambda}(\bs{\mu}_{\lambda,\tilde{f}}^V)-\F_{f,\lambda}(\bs{\mu}_{\lambda,f})
	\le 2\delta(\tilde{f}-f,V)\|\bs{\mu}_{\lambda,f}\|_{TV}+d_{\lambda}(f,\tilde{f},V,\bs{\mu}_{\lambda,f}). 
\end{split}
\end{equation}
\end{lemma}
\begin{proof}
Let $\kappa'>\kappa(V,\bs{\mu}_{\lambda,f})$ and  $\bs{\nu}\in V$ be such that
\begin{equation}
\label{eq|pRgen}
\max\left\{\|A(\bs{\mu}_{\lambda,f}-\bs{\nu})\|_{\HH}\ ,\ \left|\,\|\bs{\mu}_{\lambda,f}\|_{TV}-\|\bs{\nu}\|_{TV}  \,\right|  \right\}\leq \kappa'.
\end{equation}
For convenience, let $\tilde{\bs{\mu}}:=\bs{\mu}_{\lambda,\tilde{f}}^V$.
Since $\F_{\tilde{f},\lambda}(\tilde{\bs{\mu}})\le \F_{\tilde{f},\lambda}(\bs{\nu})$ we get from \eqref{eq|pRgen} and the definition of $\F_{\tilde{f},\lambda}$
that
\begin{equation*}
\label{eq|right_fncond_long}
\begin{split}
\F_{\tilde{f},\lambda}(\tilde{\bs{\mu}})&\le \F_{\tilde{f},\lambda}(\bs{\nu})=\F_{\tilde{f},\lambda}(\bs{\mu}_{\lambda,f}) +2\langle \tilde{f},A( \bs{\mu}_{\lambda,f}-\bs{\nu})\rangle\\&\qquad+\|A\bs{\nu}\|_{\HH}^2 -\|A\bs{\mu}_{\lambda,f}\|_{\HH}^2+\lambda(\|\bs{\nu} \|_{TV}-\|\bs{\mu}_{\lambda,f}\|_{TV})
 \\
&\le 
\F_{\tilde{f},\lambda}(\bs{\mu}_{\lambda,f})+\kappa'\,(2\|\tilde{f}\|_{\HH}+\|A\bs{\mu}_{\lambda,f}\|_{\HH}+
\|A\bs{\nu}\|_{\HH}+\lambda),
\end{split}
\end{equation*}
and since $\|A\bs{\mu}_{\lambda,f}\|_{\HH}\leq 2\|f\|_{\HH}$ 
(as $\bs{0}$ is a candidate minimizer in \eqref{eq|defcrit0}) we get that
\begin{equation*}\label{eq|right_fncond}
\F_{\tilde{f},\lambda}(\tilde{\bs{\mu}})\le\F_{\tilde{f},\lambda}(\bs{\mu}_{\lambda,f})+\kappa'\,(2\|\tilde{f}\|_{\HH}+4\|f\|_{\HH}+
\kappa'+\lambda).
\end{equation*}
As $\kappa'$ can be arbitrary close to $\kappa(V,\bs{\mu}_{\lambda,f})$, in view of \eqref{eq|defd}, this yields \eqref{eq|fncond3}.

On the left side of \eqref{eq|fncond2} we apply \eqref{eq|ineq_delta} with $\bs{\mu}=\tilde{\bs{\mu}}$ and then use the minimality of $\bs{\mu}_{\lambda,f}$, for the right side we use \eqref{eq|ineq_delta} and \eqref{eq|fncond3} after adding and subtracting the term $\F_{\tilde{f},\lambda}(\bs{\mu}_{\lambda,f})$. 
\end{proof}

\begin{lemma}\label{lemma|dlim_R-con}
Let $\{V_n\}$ be a nested increasing sequence of weak-star closed spaces in $ \mathcal{M}(S)^\N$ and let $W$ be the $R$-closure of $\bigcup_n V_n$. 
Suppose $\lambda>0$, $f\in \HH$, and $ (f_n) $ is a bounded sequence in $\HH$.  
If $\bs{\mu}\in W$, then
\begin{equation}\label{eq|dlim} \lim_{n\to\infty}d_{\lambda}(f,f_n,V_n,\bs{\mu})=0.\end{equation}

Moreover,  if  $ \lim_{n\to\infty}\delta(f_n-f,V_n)=0,$  then any subsequence of $\bs{\mu}_n:=\bs{\mu}_{f_n,\lambda}^{V_n}$ contains an $R$-convergent subsequence whose limit is a minimizer of $ \F_{ f,\lambda}$ over $W$.   In particular, if 
there is a unique minimizer $\bs{\mu}_{\lambda,f}^W$, then $\bs{\mu}_n$ converges in the $R$-topology to $\bs{\mu}_{\lambda,f}^W$.

\end{lemma}

%
%

\begin{rmk}
For the inverse magnetization problem 
\end{rmk}
\begin{proof}  
As pointed out just before Lemma~\ref{lemma|Rdense_kappa}, the $R$-density of $\bigcup_{n}V_n$ implies  
$$\lim_{n\to \infty}\kappa(\bs{\mu},V_n)=0,$$  for any $\bs{\mu}\in W$. This fact, together with the boundedness of $\|f_n\|_{\HH}$, implies \eqref{eq|dlim}.  

Now suppose $ \lim_{n\to\infty}\delta(f_n-f,V_n)=0.$  Then the boundedness of   $\|f_n\|_{\HH}$ implies that $\|\bs{\mu}_n\|$ is bounded since it cannot exceed $\|f_n\|^2_{\HH}/\lambda$
(for $\bs{0}$ is a candidate minimizer for   
$\F_{f_n,\lambda}$ over $V_n$).   Then   
\begin{equation}\label{eq|critLim}\lim_{n\to\infty}\F_{f_n,\lambda}(\bs{\mu}_n)=\min_{\bs{\mu}\in W}\F_{\bb{f},\lambda}(\bs{\mu})\end{equation} by Lemma~\ref{lemma|ineq_F_2f_V}.


Moreover, 
up to a subsequence we may assume that $\bs{\mu}_n$ converges weak-star
to some $\bs{\mu}\in\mathcal{M}(S)^\N$, and then $A\bs{\mu}_n$ tends to
$A\bs{\mu}$ in $\HH$ by the compactness of $A$. Using the lower semi-continuity of the total variation norm under weak-star convergence, we conclude in the limit 
that $\F_{\bb{f},\lambda}(\bs{\mu})\le \F_{\bb{f},\lambda}(\bs{\mu}_\lambda)$, hence $\bs{\mu}=\bs{\mu}_\lambda$ 
by \cite[Theorem 3.8]{BVH}.  Therefore $\bs{\mu}_n$ converges weak-star to $\bs{\mu}_\lambda$ and $A\bs{\mu}_n$ converges strongly to 
$A\bs{\mu}_\lambda$.  Since $\|A\bs{\mu}_n-f\|_{\HH}\to \|A\bs{\mu}_\lambda-f\|_{\HH}$ as $n\to \infty$,
it follows from \eqref{eq|critLim} that $\lim_{n\to\infty}\|\bs{\mu}_n\|_{TV}=\|\bs{\mu}_\lambda\|_{TV}$.

\end{proof}

\begin{rmk}
Let $\{V_n\}$ be a sequence of nested (finite dimensional) spaces in $ \mathcal{M}(S)^\N$ as in Lemma~\ref{lemma|dlim_R-con}, $f\in \HH$ and for every $m$, let $f^p_m=P_{ AV_m}f$, where $P_{ AV_m}$ is the orthogonal projection from $\HH$ onto   $AV_m$.
Then, for any $(m\geq)\ n$ and every $\bs{\mu}\in V_n$, $\langle f,A\bs{\mu}\rangle=\langle f^p_n,A\bs{\mu}\rangle\ (=\langle f^p_m,A\bs{\mu}\rangle)$. 
Therefore, $\F_{\bb{f},\lambda}(\bs{\mu})=\F_{f^p_n,\lambda}(\bs{\mu})\ (=\F_{f^p_m,\lambda}(\bs{\mu}))$ thus ${\mathcal{F}}_{\lambda,f}$, ${\mathcal{F}}_{f^p_n,\lambda}$ (and ${\mathcal{F}}_{f^p_m,\lambda}$) share the same minimizers over $V_n$.
Hence, using lemma \ref{lemma|dlim_R-con} by taking the sequence  $\{f_n\}$ in that lemma to be equal to $f$, if  $\bs{\mu}_n$ is minimizer of $\mathcal{F}_{f^p_n,\lambda}$ over $V_n$, then $\bs{\mu}_n\stackrel{*}{\longrightarrow}\bs{\mu}_\lambda$, $\|\bs{\mu}_n\|_{TV} \longrightarrow\|\bs{\mu}_\lambda\|_{TV}$,  and $\F_{f^p_n,\lambda}(\bs{\mu}_n)=\F_{\bb{f},\lambda}(\bs{\mu}_n)\longrightarrow\F_{\bb{f},\lambda}(\bs{\mu}_\lambda)$ as $n\to \infty$.
\end{rmk}

\begin{lemma} \label{lemma|clustermin} Let $V\subset \mathcal{M}(S)^\N$ be a closed subspace and  $f\in \HH$. Suppose $f_m$ is a bounded sequence in $ \HH$ such that
\[\lim_{m\to \infty}\langle f-f_m\,,\, Av\rangle=0,\qquad v\in V. 
\]
Then,  if $\bs{\mu}_{m}$ denotes a minimizer of $\mathcal{F}_{f_m,\lambda}$ in $V$, then any cluster point of $\bs{\mu}_{m}$ in the $R$-topology as $m\to\infty$ is a minimizer of $\mathcal{F}_{\lambda,f}$ over $V$.   
\end{lemma}

\begin{lemma} \label{lemma|doublecon} Let $\{V_n\}$ be a sequence of nested spaces in $ \mathcal{M}(S)^\N$ as in Lemma~\ref{lemma|dlim_R-con} and $ \{W_m\}$ be a sequence of nested spaces in $\HH$  whose union is dense in $\HH$.     Furthermore, suppose $S_m:\HH\to \R^m$ and $T_m:\R^m\to W_m$ 
are  linear and such that   $T_mS_m(g)=g$ for  $g\in W_m$.     Let $f\in \HH$ and $f_m=T_mS_mf$ for $m\ge 1$, then
$\lim_{m\to \infty}f_m=P_{\overline{\rnge A}}f$. 
Moreover, if $\bs{\mu}_{n,m}$ denotes a minimizer of $\mathcal{F}_{f_m,\lambda}$ in $V_n$, then any cluster point of $\bs{\mu}_{n,m}$ in the LB-topology as $m\to\infty$ is a minimizer of $\mathcal{F}_{\lambda,f}$ over $V_n$.   
\end{lemma}

\begin{proof}
Let $g=P_{\overline{\rnge A}}f$.   Then $f_m=P_{AV_m}f=P_{AV_m}g$.   
Observe that $\overline {\bigcup_m AV_m}=\overline{\rnge A}$ by the weak-star density of $\bigcup_mV_m$ in $\mathcal{M}(S)^\N$. Thus, there is a sequence $h_m\in AV_m$ such that $h_m\to g$ strongly.   Since $\|g-f_m\|\le \|g-h_m\|$, it follows that $f_m\to g$ strongly as $m\to \infty$. 

\end{proof}


\section{Convergence of the support}
\label{sec|conv_supp}
	 
\begin{rmk}
	Suppose $a_n\to 0$ (where $a_n$ is as in Lemma~\ref{lemma|equivTopo}). Then sequence $\bs{\mu}_n=P_{V_n}^{\mathcal{E}^n}(\bs{\mu})$ enjoys an additional convergence property as follows.
	Let $d_H(X,Y)$ denote the Hausdorff distance between sets $X,Y\subset\R^\N$:
	\[d_H(X,Y):=\max\left\{\sup_{x\in X} \mathrm{d}(x,Y)\,,\,
	\sup_{y\in Y} \mathrm{d}(y,X)\right\}.\]
	Letting $J_n=\{k\in I_n:\bs{\mu}(E^n_k)\neq0\}$, we get that supp$(\bs{\mu}_n)=\overline{\bigcup_{k\in J_n}\text{supp}(\nu^n_k)}$, and supp$(\bs{\mu})=\overline{\bigcup_{k\in J_n}\text{supp}(\bs{\mu})\cap E^n_k}$.
	Then, noting that for any $X,Y\subset\R^\N$, $d_H(\overline{X},\overline{Y})=d_H(X,Y)$ it follows that:
	\begin{align*}
		d_H&(\text{supp}(\bs{\mu}_n),\text{supp}(\bs{\mu}))
		=d_H\left(\bigcup_{k\in J_n}\text{supp}(\nu^n_k),\bigcup_{k\in J_n}\text{supp}(\bs{\mu})\cap E^n_k\right)\\
		&=\max\left\{\sup_{k\in J_n}\left\{\sup_{x\in\text{supp}(\nu^n_k)}d\left(x,\bigcup_{k\in J_n}\text{supp}(\bs{\mu})\cap E^n_k\right)\right\}
		,\sup_{k\in J_n}\left\{\sup_{x\in\text{supp}(\bs{\mu})\cap E^n_k}d\left(x,\bigcup_{k\in J_n}\text{supp}(\nu^n_k)\right)\right\}\right\}\\
		&\leq\max\left\{\sup_{k\in J_n}\left\{\sup_{x\in\text{supp}(\nu^n_k)}d(x,\text{supp}(\bs{\mu})\cap E^n_k)\right\}
		,\sup_{k\in J_n}\left\{\sup_{x\in\text{supp}(\bs{\mu})\cap E^n_k}d(x,\text{supp}(\nu^n_k))\right\}\right\}\\
		&=\sup_{k\in J_n}d_H\left(\text{supp}(\nu^n_k),\text{supp}(\bs{\mu})\cap E^n_k\right).
	\end{align*} 
	Therefore $d_H($supp$(\bs{\mu}_n),$supp$(\bs{\mu}))\leq a_n\to0$. 
	
	The convergence of $\mathrm{supp}(\bs{\mu}_n)$ to
	$\mathrm{supp}(\bs{\mu})$ with respect to the Hausdorff distance,
	however,  is  not shared by all \textit{R}-convergent sequences.
	Still,  note that  if $\bs{\mu}\in\mathcal{M}(S)^\N$ and the sequence $\bs{\mu}_n$ converges to $\bs{\mu}$ in the $R$-topology, then for any continuity set $W$ of $|\bs{\mu}|$, it follows from the Portmanteau theorem \cite{Bil} that $|\bs{\mu}_n|(W)\to|\bs{\mu}|(W)$.
	In particular, if $W$ is a continuity set of $|\bs{\mu}|$ 
	disjoint from supp$(\bs{\mu})$ (this happens for instance if $\overline{W}\cap
	\mathrm{supp}(\bs{\mu})=\emptyset$), then $|\bs{\mu}_n|(W)\to0$.
\end{rmk}

\begin{theorem}
	\label{thm|gensupp}
	Let $V$ be a finite dimensional GSM space with (finite) index set $I$ as in  \eqref{eq|defGSM}.  For $\lambda>0$ and $f\in L^2(Q)$,
	write $\bs{\mu}_{\lambda,f}$ (rather than
	$\bs{\mu}_{\lambda,f}^V$)  to indicate the minimizer
	of \eqref{eq|mulamfVreg} in $V $ and let  $\bs{\mu}_{\lambda,f}=\sum_{k\in I}\bb{m}_k^{\lambda,f}\nu_k$ for the corresponding decomposition of $\bs{\mu}_{\lambda,f}$.
	Assume that $A:V\to L^2(Q)$ is injective, and for $\bs{\mu}\in V$ define 
	$$L_\alpha(\bs{\mu}):=\{k\in I\ :\ |\langle\nu_k,A^*(f-A\bs{\mu})\rangle|=\alpha\}.$$ 
	Then,  there is a dense subset $\mathcal{O}$ of $L^2(Q)\times\R^+$
	such that \begin{equation}
		\label{eq|genl2}
		L_{\lambda/2}(\bs{\mu}_{\lambda,f})=\{k\in I\ :\ \bb{m}_k^{\lambda,f}\neq0\},
		\qquad (f,\lambda)\in\mathcal{O}.
	\end{equation}
\end{theorem}
\begin{proof}
	For $K\subset I$, consider the subspace of  $V$ given by
	$      V_K:=\{\sum_{k\in K} {\bf m}_k \nu_k:\,  {\bf m}_k\in\R^\N\}$,
	and identify the latter with ${(\R^\N)}^{|K|}$ upon regarding
	$\sum_{k\in K} {\bf m}_k \nu_k$ as enumerating the ${\bf m}_k$ in some
	conventional order; here, $|K|$ is the cardinality of $K$. We shall denote with
	$V^\sharp_K\subset V$ the open subset of those $\sum_{k\in K} {\bf m}_k \nu_k$
	such that ${\bf m}_k\neq0$ for all $k\in K$.
	If $K\neq I$ pick $k_0\in I\setminus K$, and define a map $G_K:L^2(Q)\times\R^+\times V_k^\sharp\to \R^{3|K|+1}$ by
	letting for $\bs{\mu}=\sum_{k\in K} {\bf m}_k \nu_k$:
	\[G_K(f,\lambda,\bs{\mu}):=\left(\Bigl(\langle\nu_k,A^*(f-A\bs{\mu})\rangle-
	\frac{\lambda}{2}\frac{{\bf m}_k}{|{\bf m}_k|}\Bigr)_{k\in K}\,,\,
	\Bigl|\langle\nu_{k_0},A^*(f-A\bs{\mu})\rangle\Bigr|^2-\frac{\lambda^2}{4}\right).\]
	When $K=I$, we proceed in the same way except that we omit the place containing $k_0$, and then $G$ is valued in $\R^{3|I|}$.
	Clearly, $G_K$ is a $C^\infty$-smooth map; {\it i.e.}, it is indefinitely differentiable
	(in  the sense of Fr\'echet) as  a map between Banach (even Hilbert) spaces.
	Computing the differential with respect to the first variable, we get that
	\[
	D_1G_K(f,\lambda,\bs{\mu})(g)=
	\left(\Bigl(\langle\nu_k,A^*(g)\rangle\Bigr)_{k\in K}\,,\,
	2\langle\nu_{k_0},A^*(f-A\bs{\mu})\rangle\cdot
	\langle\nu_{k_0},A^*(g)\rangle\right),\quad g\in L^2(Q),
	\]
	and by definition of the adjoint operator the above expression is equal to
	$(\mathfrak{X},\mathfrak{y})$ with
	\[\mathfrak{X}:=
	\Bigl(
	(\langle A({\bf e}_1\nu_k),g\rangle_{L^2(Q)},\langle A({\bf e}_2\nu_k),g\rangle_{L^2(Q)},\langle A({\bf e}_3\nu_k),g\rangle_{L^2(Q)})^T
	\Bigr)_{k\in K}
	\]
	while
	\[\mathfrak{y}:=
	2\langle\nu_{k_0},A^*(f-A\bs{\mu})\rangle\cdot
	(\langle A({\bf e}_1\nu_{k_0}),g\rangle_{L^2(Q)},\langle A({\bf e}_2\nu_{k_0}),g\rangle_{L^2(Q)},
	\langle A({\bf e}_3\nu_{k_0}),g\rangle)^T,
	\]
	with $({\bf e}_i)_{1\leq i\leq3}$ to mean the canonical basis of $\R^\N$.
	Now, since the functions $A({\bf e}_i \nu_k)$
	are independent in
	$L^2(Q)$ when $k$ ranges over $I$ and $i$ over $\{1,2,3\}$, by the injectivity of $A$ on $V$, we find that $(\mathfrak{X},\mathfrak{y})$ ranges over all
	of $\R^{3|K|+1}$ as $g$ ranges over $L^2(Q)$ provided that
	$\langle\nu_{k_0},A^*(f-A\bs{\mu})\rangle$ is nonzero, because in a Hilbert space one can always find an element with prescribed scalar products against a
	finite family of independent vectors.
	Therefore $DG_K$ is {\it a fortiori} surjective at every $(f,\lambda,\bs{\mu})$
	such that $G_K(f,\lambda,\bs{\mu})=0$ and hence,
	by the transversal density theorem \cite[Theorem 19.1]{AbRo},
	the set $\mathcal{R}_K$ of those
	$(f,\lambda)$ such that the partial map $G(f,\lambda,\cdot):V^\sharp_K
	\to\R^{3|K|+1}$ is transverse to the zero-dimensional submanifold $\{0\}\subset\R^{3|K|+1}$ is residual in the Baire sense. For such $(f,\lambda)$
	the set of those ${\bs \mu}$ meeting $G(f,\lambda,\bs{\mu})=0$ is a
	submanifold of $V^\sharp_K$ of codimension $3|K|+1$ in a space of dimension
	$3|K|$, therefore it is empty. Taking the intersection of all $\mathcal{R}_K$
	over $K\subset I$ and $k_0\in I\setminus K$ (this a finite family),
	we get a residual set  $\mathcal{R}$ such that
	the equality in \eqref{eq|genl2} holds. This shows that the set
	$\mathcal{O}$ for which  \eqref{eq|genl2} holds is true is dense.
\end{proof}

\begin{prop}\label{prop|glambda}
	Let $V$ be a GSM space and put $g_{\lambda}^V =| A^*(f-A(\bs{\mu}_\lambda^V))|$ as well as $g_{\lambda}  =| A^*(f-A(\bs{\mu}_\lambda ))|$.
	If  $\|g_{\lambda}^V -g_{\lambda}\|_{L^\infty(S)}<\epsilon$  for some $\epsilon>0$ and $g_\lambda(x)<\lambda/2-\epsilon$ for   $\nu_k$-a.e $x$ for some $k\in I$, then $\bs{m}_k=0$.  
\end{prop}
\begin{proof}
	We remark that the hypotheses imply  $g_\lambda^V(x)<\lambda/2$ for   $\nu_k$-a.e $x$, and hence $$\left |\langle \nu_k , A^*(f-A\bs{\mu_\lambda^V})\rangle\right| < \frac{\lambda}{2},$$ and  thus $\bs{m}_k=0$ by Corollary~\ref{coro|CPV}.  
\end{proof}

\begin{lemma}
	Let $\lambda>0$, for any $\bs{\mu}\in\mathcal{M}(S)^\N$, $L_\alpha(\bs{\mu}):=\{x\in S\ :\ |A^*A(\bs{\mu}-\bs{\mu}_{\lambda,A\bs{\mu}})|=\alpha\}$ and $T:=\{\bs{\mu}\in\mathcal{M}(S)^\N\ |\ L_{\lambda/2}(\bs{\mu})=\supp\bs{\mu}_{\lambda,A\bs{\mu}}\}$.
	
	Then, for every  $\bs{\mu}\in\mathcal{M}(S)^\N$ there exist a $\bs{\nu}_\bs{\mu}\in\mathcal{M}(S)^\N$ such that $\bs{\mu}$, $\bs{\nu}_\bs{\mu}$ are
	mutually singular, $\supp\bs{\mu}\cup\supp\bs{\nu}_\bs{\mu}=L_{\lambda/2}(\bs{\mu})$ and, for any $\epsilon>0$, $\bs{\mu}+\epsilon\bs{\nu}_\bs{\mu}\in T$.
	
	Furthermore, $T$ is dense in $\mathcal{M}(S)^\N$ for any topology that makes $\mathcal{M}(S)^\N$ a topological vector space.
\end{lemma}
\begin{proof}
	Note that \eqref{eq|generalCPV} implies that $\supp\bs{\mu}_{\lambda,A\bs{\mu}}\subset L_{\lambda/2}(\bs{\mu})$.
	Let $\{x_i\}_{i=1}^\infty$ be a dense subset of $L_{\lambda/2}(\bs{\mu})\setminus\supp\bs{\mu}_{\lambda,A\bs{\mu}}$ and $\nu:=\sum_i\frac{1}{2^{i}}\delta_{x_i}$.
	If we define $\bs{\nu}_\bs{\mu}\in\mathcal{M}(S)^\N$ by \[d\bs{\nu}_\bs{\mu}=\frac{2}{\lambda}A^*A(\bs{\mu}-\bs{\mu}_{\lambda,A\bs{\mu}})d\nu,\]
	then $\bs{\mu}$, $\bs{\nu}_\bs{\mu}$ are mutually singular and $\supp\bs{\mu}\cup\supp\bs{\nu}_\bs{\mu}=L_{\lambda/2}(\bs{\mu})$.
	Now, for any $\epsilon>0$, since $A^*A(\bs{\mu}+\epsilon\bs{\nu}_\bs{\mu}-(\bs{\mu}_{\lambda,A\bs{\mu}}+\epsilon\bs{\nu}_\bs{\mu}))=A^*A(\bs{\mu}-\bs{\mu}_{\lambda,A\bs{\mu}})$, by using \eqref{eq|generalCPV} we can see that $\bs{\mu}_{\lambda,A(\bs{\mu}+\epsilon\bs{\nu}_\bs{\mu})}=\bs{\mu}_{\lambda,A\bs{\mu}}+\epsilon\bs{\nu}_\bs{\mu}$.
	Therefore $L_\alpha(\bs{\mu}+\epsilon\bs{\nu}_\bs{\mu})=L_\alpha(\bs{\mu})$ and hence  $\bs{\mu}+\epsilon\bs{\nu}_\bs{\mu}\in T$.
	Finally, by continuity of addition and product by a scalar, the second statement of the lemma now follows.
\end{proof}

\begin{prop}\label{prop|dist}
	Let $\lambda>0$, $f\in  \HH$, and $V_n\subset\mathcal{M}(S)^\N$ be a nested increasing sequence of weak-star closed GSM spaces.
	Let $\{E^n_k\}_{k\in I_n}$ be a compatible Borel partition of $S$ for $V_n$ as a GSM space and let $a_n$ be the supremum of the diameters of the $E^n_k$.
	Further assume that $a_n\to 0$ as $n\to\infty$.
	If we let $L_\alpha=\{x\in S\ :\ |A^*(f-A\bs{\mu}_{\lambda,f})|=\alpha\}$, then,
	\begin{equation}\label{eq|dist_level}
		\sup_{x\in\supp(\bs{\mu}_{\lambda,f}^{V_n})}d(x,L_{\lambda/2})\to 0, \text{ as } n\to\infty,
	\end{equation}
	and
	\begin{equation}\label{eq|dist_supp}
		\sup_{x\in\supp(\bs{\mu}_{\lambda,f})} d(x,\supp(\bs{\mu}_{\lambda,f}^{V_n}))\to 0, \text{ as } n\to\infty.
	\end{equation}
	Furthermore, any compact set $C\subset S$ that is a cluster point of $\{\supp(\bs{\mu}_{\lambda,f}^{V_n})\}_n$ under the Hausdorff distance, is such that $\supp(\bs{\mu}_{\lambda,f})\subset C \subset L_{\lambda/2}$.
\end{prop}
\begin{proof}
	Note that, by lemma \ref{lemma|prevexample}, $\bigcup_{n}V_n$ is R-dense in $\mathcal{M}(S)^\N$.
	Let $g_\lambda$ and the $g_{\lambda}^V$ be as in proposition \ref{prop|glambda}.
	Recall that $A$ is compact and that $A^*$ is strongly continuous.
	Thus $A^*A$ is compact as well and hence weak to strong sequentially continuous.
	By lemma \ref{lemma|dlim_R-con}, $\bs{\mu}_{\lambda,f}^{V_n}\stackrel{R}{\longrightarrow}\bs{\mu}_{\lambda,f}$, then, since  $A^*A:\mathcal{M}(S)^\N\to C(S)^\N$ is weak star to strong sequentially continuous, we have that $g_{\lambda}^{V_n}\to g_\lambda$ uniformly as $n\to\infty$.
	Take $\varepsilon>0$ and let $\tilde{N}>0$ be the such that $n\geq \tilde{N}$ implies that $\|g_{\lambda}^{V_n}- g_\lambda\|_\infty<\varepsilon$.
	Let $L^+_\alpha=\{x\in S\ :\ |A^*(f-A\bs{\mu}_{\lambda,f})|\geq\alpha\}$.
	Then, by proposition \ref{prop|glambda}, for $n>\tilde{N}$, and for any $k$ such that $|\bs{\mu}_{\lambda,f}^{V_n}|(E^n_k)>0$, we also have that $L^+_{\lambda/2-\varepsilon}\cap E^n_k\neq\emptyset$.
	Thus, for $|\bs{\mu}_{\lambda,f}^{V_n}|$-a.e.  $z\in\supp(\bs{\mu}_{\lambda,f}^{V_n})$, we can find a $k\in I_n$ and a $y_z\in S$ such that both $z\in E^n_k$ and $y_z\in L^+_{\lambda/2-\varepsilon}\cap E^n_k$.
	Hence $d(z,L_{\lambda/2})\leq |z-y_z|+d(y_z,L_{\lambda/2})$ which implies that
	\begin{equation*}
		\sup_{x\in\supp(\bs{\mu}_{\lambda,f}^{V_n})}d(x,L_{\lambda/2}) 
		\leq a_n + \sup_{x\in L^+_{\lambda/2-\varepsilon}}d(x,L_{\lambda/2}).
	\end{equation*}
	Finally, by continuity of $g_\lambda$ and the fact that $a_n\to 0$ as $n\to\infty$, \eqref{eq|dist_level} follows.
	
	Now, for any $n>0$ and $k\in I_n$, take $x^n_k\in E^n_k$.
	Since for any $n>0$, the family $\{\B_{a_n}(x^n_k)\}_{k\in I_n}$ covers the compact set $\supp(\bs{\mu}_{\lambda,f})$, we can find a finite subset $J_n\subset I_n$ such that $\{\B_{a_n}(x^n_k)\}_{k\in J_n}$ is a subcover of $\supp(\bs{\mu}_{\lambda,f})$.
	For every pair $k,n$ that satisfies $|\bs{\mu}_{\lambda,f}|(E^n_k)>0$ we can find a $\phi^n_k\in C_0(\B_{a_n}(x^n_k))$ such that $\langle |\bs{\mu}_{\lambda,f}|,\phi^n_k\rangle>0$.
	Then, by $R$-convergence, there exists a $\tilde{N}>0$ such that for every $n>\tilde{N}$ and $k\in I_n$ that satisfies $|\bs{\mu}_{\lambda,f}|(E^n_k)>0$, we have that $\langle |\bs{\mu}_{\lambda,f}^{V_n}|,\phi^n_k\rangle>0$ and thus $E^n_k\subset\supp(\bs{\mu}_{\lambda,f}^{V_n})$.
	Therefore, for $n>\tilde{N}$, and $x\in\supp(\bs{\mu}_{\lambda,f})$, we have that $x\in\B_{a_n}(x^n_k)$ for some $k\in J_n$, and thus $d(x,\supp(\bs{\mu}_{\lambda,f}^{V_n}))\leq d(x,x^n_k)+d(x^n_k,\supp(\bs{\mu}_{\lambda,f}^{V_n}))=d(x,x^n_k)+0<a_n$, which implies \eqref{eq|dist_supp}.
	

\end{proof}



\appendix

\section{}

\begin{lemma}\label{lemma|limitsum}
	Assume that 
	$\|\bs{\mu}_n\|_{TV}\to c_1>0$ and $\|\bs{\nu}_n\|_{TV}\to c_2>0$, 
	and that there exist a family of Borel sets $U_n$, such that $|\bs{\mu}_n|(U_n)\to 0$ and $|\bs{\nu}_n|(S\setminus U_n)\to 0$.
	Then $\|\bs{\mu}_n+\bs{\nu}_n\|_{TV}\to c_1+c_2$.
\end{lemma}
\begin{proof}
	Let $\delta>0$ be such that $c_1,c_2>\delta$.
	Then, for $n$ sufficiently large, there exists $a_n>0$ such that $|\bs{\mu}_n|(U_n),|\bs{\nu}_n|(S\setminus U_n)<a_n<\delta/2$ and $a_n\to0$.
	For every $n$ we have that $$\|\bs{\mu}_n+\bs{\nu}_n\|_{TV}=
	{\Big\|\bs{\mu}_n\mathcal{b}{U_n}+\bs{\nu}_n\mathcal{b}{U_n}\Big\|_{TV}} + {\Big\|\bs{\mu}_n\mathcal{b}(S\setminus U_n)+\bs{\nu}_n\mathcal{b}(S\setminus U_n)\Big\|_{TV}}.$$
	Next, for $n$ sufficiently large we have that $|\bs{\mu}_n|(U_n),|\bs{\nu}_n|(S\setminus U_n)<a_n<\delta/2<|\bs{\mu}_n|{(S\setminus U_n)},|\bs{\nu}_n|(U_n)$, hence,
	\begin{align*}
		\|\bs{\mu}_n\|_{TV}-2a_n+\|\bs{\nu}_n\|_{TV}-2a_n 
		&<\big(|\bs{\mu}_n|(S\setminus U_n)-|\bs{\nu}_n|(S\setminus U_n)\big) + \big(|\bs{\nu}_n|(U_n)-|\bs{\mu}_n|(U_n)\big) \\
		&= \big| \|\bs{\mu}_n\mathcal{b}(S\setminus U_n)\|_{TV} - \|\bs{\nu}_n\mathcal{b}(S\setminus U_n)\|_{TV} \big| + ... \\
		& \qquad + \big|\|\bs{\nu}_n\mathcal{b}{U_n}\|_{TV} - \|\bs{\mu}_n\mathcal{b}{U_n}\|_{TV} \big|
		\\
		&\leq \|\bs{\mu}_n+\bs{\nu}_n\|_{TV} \\
		&\leq \|\bs{\mu}_n\|_{TV} + \|\bs{\nu}_n\|_{TV}.
	\end{align*}
	Therefore, the lemma follows.
\end{proof}

Recall that if $V_1\subset V_2\subset\cdots$ is a nested sequence of subsets of 
$\mathcal{M}(S)^3$ whose union is $R$-dense, then
$\lim_n\kappa(\bs{\mu},V_n)=0$ for each $\bs{\mu}\in \mathcal{M}(S)^3$.
In fact, much more is true:
\begin{lemma}\label{lemma|Rdense_kappa}
If $\{V_n\}$ is a nested sequence of weak-star closed spaces in $ \mathcal{M}(S)^3$ whose union is $R$-dense, then $\kappa(V_n)\to 0$ as $n\to \infty$.
\end{lemma}
\begin{proof}

Assume for a contradiction that $\kappa(V_n)>\delta>0$ for all $n$, and let $\bs{\mu}_n\in\mathcal{M}(S)^3$ be such that $\kappa(\bs{\mu}_n,V_n)\geq\delta$ and $\|\bs{\mu}_n\|_{TV}=1$.
Restricting to a subsequence, we may assume without loss of generality that there exists a $\bs{\mu}\in\mathcal{M}(S)^3$, with $\|\bs{\mu}\|_{TV}\leq 1$,
such that $\bs{\mu}_n$ converges weak-star to $\bs{\mu}$, by the BanachAlaoglu theorem.
Since the union of the $V_n$ is $R$-dense on $\mathcal{M}(S)^3$, we can find a sequence $\bs{\nu}_n$, with $\bs{\nu}_n\in V_n$, such that $\bs{\nu}_n\to\bs{\mu}$ in the $R$-topology.
Then, the compactness of $A$ implies that for $n$ sufficiently large 
\begin{equation}
\label{eq|majRTV}
\left|\|\bs{\mu}_n\|_{TV}-\|\bs{\nu}_n\|_{TV}\right|\geq\kappa(\bs{\mu}_n,V_n)\geq\delta.
\end{equation}
Because $\|\bs{\nu}_n\|_{TV}\to\|\bs{\mu}\|_{TV}$ and $\|\bs{\mu}_n\|_{TV}=1$, \eqref{eq|majRTV} implies that $\|\bs{\mu}\|_{TV}\leq 1-\delta$.

Remember that $S$ is assumed uncountable. 
Then, since $|\bs{\mu}|$ is finite, there exists an $x\in S$ which is a limit point of $S$ such that $|\bs{\mu}|(\{x\})=0$.  As $\lim_{r\to 0}|\bs{\mu}|(\B(x,r))=|\bs{\mu}|(\{x\})=0$, for  each $n$ there is
  $r_n>0$  such that  $|\bs{\mu}|(\B(x,r_n))<1/n$, and  since $x$ is a limit point we may require in addition that there exists $x_n\in S\cap\B(x,r_n/2)\setminus\B(x,r_n/3)$. Moreover,
by the finiteness of $|\bs{\mu}|$ again, we may choose $r_n$ to satisfy $|\bs{\mu}|(\partial\B(x,r_n))=0$ for each $n$; i.e., such that
$\B(x,r_n)$ is a continuity set of $|\bs{\mu}|$.
As Lemma \ref{lemma|equivTopo} implies that $|\bs{\nu}_n|\to|\bs{\mu}|$ in the weak-star sense, by the Portmanteau theorem \cite{Bil} there exists $m_n$ such that $|\bs{\nu}_{m_n}|(\B(x,r_n))<1/n$.
Next, we will construct a sequence $\bs{\varsigma}_n$ with $\bs{\varsigma}_n\in V_{l_n}$ for some suitable $l_n\geq n$, such that $|\bs{\varsigma}_n|(S\setminus\B(x,r_n))<1/n$ and $\bs{\varsigma}_n\to0$ weak-star,  while
$\|\bs{\varsigma}_n\|_{TV}\to 1-\|\bs{\mu}\|_{TV}$ as $n\to\infty$.

Fix $\bs{w}\in\R^3$, $|\bs{w}|=1$, to let $\bs{a}:=\frac{1}{2}{(1-\|\bs{\mu}\|_{TV})}{\delta_x}\bs{w}$ and $\bs{b}_n:=-\frac{1}{2}{(1-\|\bs{\mu}\|_{TV})}{\delta_{x_n}}\bs{w}$.
Note that $\bs{a}+\bs{b}_n \to 0$ weak-star as $n\to\infty$.
Pick a sequence $(\overline{\bs{a}}_m)_{m\in\mathbb{N}}$ and
for each $n$ a sequence 
 $(\overline{\bs{b}}_m ^n)_{m\in\mathbb{N}}$ such that 
$\overline{\bs{a}}_m\in V_m$,  $\overline{\bs{b}}_m ^n\in V_m$,
$\overline{\bs{a}}_m\to\bs{a}$ and
$\overline{\bs{b}}_m^n\to\bs{b}_n$ as $m\to\infty$, in the $R$-topology.
Thus, Lemma \ref{lemma|equivTopo} implies that $|\overline{\bs{a}}_m|\to|\bs{a}|$ and $|\overline{\bs{b}}_m^n|\to|\bs{b}_n|$ weak-star as $m\to\infty$.



By Urysohn's lemma, for each $n$ there is a continuous $f_n:S\to[-1,1]$ with 
support in $B(x,3r_n/4)$ such that $f_n(x)=-f_n(x_n)=1$. 
Hence, by weak-star convergence, the limit of $\int f_n\bs{w} \cdot d(\overline{\bs a}_m +\overline{\bs b}^n_m)$ as $m\to\infty$ is $1-\|\bs{\mu}\|_{TV}$. 
Consequently, $\liminf_m\|\overline{\bs a}_m +\overline{\bs b}^n_m\|_{TV}$ 
is at least $1-\|\bs{\mu}\|_{TV}$, and since 
\[
\limsup_m\|\overline{\bs a}_m +\overline{\bs b}^n_m\|_{TV}\leq 
\lim_m \|\overline{\bs a}_m\|_{TV} + \lim_m \|\overline{\bs b}^n_m\|_{TV}=1-\|\bs{\mu}\|_{TV},\]
we find that 
$\lim_m\|\overline{\bs a}_m +\overline{\bs b}^n_m\|_{TV}= 1-\|\bs{\mu}\|_{TV}$.

Therefore, $\overline{\bs a}_m +\overline{\bs b}^n_m\to\bs{a}+\bs{b}_n$ in the $R$-topology as $m\to\infty$, so there exists $k_n\geq n$ such that
$$
\Big|\|\overline{\bs{a}}_{m}+\overline{\bs{b}}_{m}^n\|_{TV} - \|\bs{a}+\bs{b}_n\|_{TV}\Big|
\leq\mathfrak{d}_R(\overline{\bs{a}}_{m}+\overline{\bs{b}}_{m}^n,\bs{a}+\bs{b}_n)<1/n \ \text{ for } m\geq k_n.
$$

Lemma \ref{lemma|equivTopo} now implies that $|\overline{\bs{a}}_{m}+\overline{\bs{b}}_{m}^n|\to|\bs{a}+\bs{b}_n|$ weak-star as $m\to\infty$, and
since $S\setminus\B(x,r_n)$ is a continuity set of $|\bs{a}+\bs{b}_n|$
we get on  using again the Portmanteau theorem 
while possibly replacing $k_n$ by some $l_n\geq k_n$ that $|\overline{\bs{a}}_{l_n}+\overline{\bs{b}}_{l_n}^n|(S\setminus\B(x,r_n))<1/n$.
%
Noting that $\|\bs{a}+\bs{b}_n\|_{TV}=1-\|\bs{\mu}\|_{TV}$, 
we see that $\bs{\varsigma}_n:=\overline{\bs{a}}_{l_n}+\overline{\bs{b}}_{l_n}^n$ has the desired properties.

Now, if $\|\bs{\mu}\|_{TV}=0$ let $\bs{v}_n:=\bs{\varsigma}_n$, otherwise let $\bs{v}_n:=\bs{\nu}_{m_n}+\bs{\varsigma}_n$.
From Lemma \ref{lemma|limitsum}, it follows that $\|\bs{v}_n\|_{TV} \to 1$.
However, using  the weak-star convergence of $\bs{\varsigma}_n$ to zero
and the compactness of $A$ again, we get that $\left|\|\bs{\mu}_n\|_{TV}-\|\bs{v}_n\|_{TV}\right|\geq\kappa(\bs{\mu}_n,V_n)\geq\delta$, which is
a contradiction since $\|\bs{\mu}_n\|_{TV}=1$.

\end{proof}


 \bibliographystyle{abbrv}

\bibliography{C2D}

\begin{thebibliography}{1}

\bibitem{AbRo}
R.~Abraham and J.~Robbin.
\newblock {\em Transversal mappings and flows}.
\newblock W. A. Benjamin, 1967.

\bibitem{BVH}
L.~Baratchart, C.~Villalobos~Guill{\'e}n, and D.~P. Hardin.
\newblock Inverse potential problems in divergence form for measures in the
  plane.
\newblock {\em ESAIM: COCV}, 27, 2021.

\bibitem{BVHNS}
L.~Baratchart, C.~Villalobos~Guill{\'e}n, D.~P. Hardin, M.~C. Northington, and
  E.~B. Saff.
\newblock Inverse potential problems for divergence of measures with total
  variation regularization.
\newblock {\em Foundations of Computational Mathematics}, Nov 2019.

\bibitem{Bil}
P.~Billingsley.
\newblock {\em Convergence of probability measures}.
\newblock Wiley Series in Probability and Statistics: Probability and
  Statistics. John Wiley \& Sons, Inc., New York, second edition, 1999.
\newblock A Wiley-Interscience Publication.

\bibitem{BrePikk}
K.~Bredies and H.~K. Pikkarainen.
\newblock Inverse problems in spaces of measures.
\newblock {\em ESAIM COCV}, 19:190--218, 2013.

\bibitem{duren_1970}
P.~L. Duren.
\newblock {\em Theory of hp spaces}.
\newblock Academic Press, 1970.

\bibitem{ekeland_turnbull_1983}
I.~Ekeland and T.~Turnbull.
\newblock {\em Infinite-dimensional optimization and convexity}.
\newblock University of Chicago Press, 1983.

\bibitem{Rudin}
W.~Rudin.
\newblock {\em Functional Analysis}.
\newblock Mc Graw-Hill, 1991.

\end{thebibliography}

\end{document}